\documentclass{elsarticle-ppn}

\usepackage[T1]{fontenc}
\usepackage[utf8]{inputenc}
\usepackage{amsmath,amsfonts,amsthm,amssymb,amsxtra,bbm}
\usepackage{fourier,setspace,graphicx,color,pdflscape}
\usepackage{textcomp}
\usepackage{color}
\usepackage[colorlinks=true]{hyperref}
\hypersetup{urlcolor=blue, linkcolor=blue, citecolor=red, anchorcolor=blue]}
\usepackage{mathrsfs}
\usepackage{mathtools}
\usepackage{setspace}
\usepackage{lineno,hyperref}

\usepackage{etoolbox}
\newcounter{taggedeq}
\pretocmd{\equation}{\stepcounter{taggedeq}}{}{}

\newtheorem{theorem}{Theorem}
\newtheorem{proposition}[theorem]{Proposition}
\newtheorem{lemma}[theorem]{Lemma}
\newtheorem{corollary}[theorem]{Corollary}

\newcommand{\R}{{\mathbb R}}
\newcommand{\N}{{\mathbb N}}

\newcommand{\E}{{\mathscr E}}
\newcommand{\I}{{\mathscr I}}
\newcommand{\J}{{\mathscr J}}
\newcommand{\K}{{\mathscr K}}
\renewcommand{\L}{{\mathscr L}}

\newcommand{\be}[1]{\begin{equation}\label{#1}}

\newcommand{\ee}{\end{equation}}
\renewcommand{\(}{\left(}
\renewcommand{\)}{\right)}

\newcommand{\irdy}[1]{\int_{\R^d}{#1}\,dy}

\newcommand{\irdg}[1]{\int_{\R^d}{#1}\,d\gamma}
\newcommand{\nrmG}[2]{\|#1\|_{\mathrm L^{#2}(\R^d,d\gamma)}}
\newcommand{\Cst}\eta
\newcommand{\CGS}{\mathrm C}
\newcommand{\taus}{s}

\makeatletter
\@namedef{subjclassname@2020}{\textup{2020} Mathematics Subject Classification}
\makeatother
\newcommand{\msc}[1]{\href{https://zbmath.org/classification/?q=cc:#1}{#1}}

\renewenvironment{thebibliography}[1]{\begin{oldthebibliography}{#1}\setlength{\itemsep}{0em}\setlength{\parskip}{0em}}{\end{oldthebibliography}}

\begin{document}
\begin{frontmatter}

\title{Stability for the logarithmic Sobolev inequality}

\author[IST]{Giovanni Brigati}
\ead{giovanni.brigati@ist.ac.at}
\author[CEREMADE]{Jean Dolbeault\corref{cor}}
\ead{dolbeaul@ceremade.dauphine.fr}
\author[LJLL]{Nikita Simonov}
\ead{nikita.simonov@sorbonne-universite.fr}
\cortext[cor]{Corresponding author}
\address[IST]{Institute of Science and Technology Austria (ISTA), Am Campus 1, 3400 Klosterneuburg, Austria}
\address[CEREMADE]{CEREMADE (CNRS UMR n$^\circ$~7534), PSL University, Universit\'e Paris-Dauphine,\newline Place de Lattre de Tassigny, 75775 Paris 16, France}
\address[LJLL]{LJLL (CNRS UMR n$^\circ$~7598), Sorbonne Universit\'e, 4 place Jussieu, 75005 Paris, France}

\begin{abstract} This paper is devoted to stability results for the Gaussian logarithmic Sobolev inequality, with explicit stability constants. \end{abstract}

\begin{keyword}
logarithmic Sobolev inequality, stability, log-concavity, heat flow, entropy, carr\'e du champ, Poincar\'e inequality.
\MSC[2020] Primary: \msc{39B62}; Secondary: \msc{47J20}, \msc{49J40}, \msc{35A23}, \msc{35K85}.

\end{keyword}
\end{frontmatter}

\section{Introduction}\label{Sec:Intro}

\subsection{Main result}
Let $d \geq 1,$ and let us consider the \emph{Gaussian logarithmic Sobolev inequality}
\be{LSIg}\tag{LSI}
\nrmG{\nabla u}2^2\ge\frac12\irdg{|u|^2\,\log\(\frac{|u|^2}{\nrmG u2^2}\)}\quad\forall\,u\in\mathrm H^1(\R^d,d\gamma)
\ee
where $d\gamma=\gamma(x)\,dx$ is the normalized Gaussian probability measure with density
\[
\gamma(x)=(2\,\pi)^{-\frac d2}\,e^{-\frac12\,|x|^2}\quad\forall\,x\in\R^d\,.
\]
According to~\cite{MR1132315,bakry1985diffusions}, equality in~\eqref{LSIg} is achieved by any function in the manifold
\[
\mathscr M:=\big\{w_{a,c}\,:\,(a,c)\in\R^d\times\R\big\}\quad\mbox{where}\quad w_{a,c}(x)=c\,e^{a\cdot x}\quad\forall\,x\in\R^d
\]
and only by functions in $\mathscr M$.

The issue of \emph{stability} in functional inequalities is either (a) to estimate the distance to~$\mathscr M$ by the deficit in~\eqref{LSIg}, \emph{i.e.}, by the difference of the two sides in the inequality, or (b) to obtain an \emph{improved inequality}, \emph{i.e.}, an improved constant in the inequality, under appropriate conditions. In fact, (b) amounts to establish (a) but for a restricted class of functions. More details will be given later. Our main result is in the spirit of (b). By two-homogeneity of~\eqref{LSIg}, we may consider functions in
\[
\mathscr H:=\left\{u\in\mathrm H^1(\R^d,d\gamma)\,:\,\nrmG u2=1\right\}
\]
without loss of generality and state an \emph{improved~\eqref{LSIg} inequality} as follows.
\begin{theorem}\label{Thm:Main} Let $d\ge1$. For any $\varepsilon>0$ and $\CGS>0$, there is an explicit $\Cst\in(0,1)$ such that
\[
\(1-\Cst\)\nrmG{\nabla u}2^2\ge\frac12\irdg{|u|^2\,\log|u|^2}
\]
for any $u\in\mathcal H$ such that
\be{Moment}
\irdg{x\,|u|^2}=0
\ee
and
\be{gautail}
\iint_{\R^d\times\R^d}|u(x)|^2\,|u(y)|^2\,e^{\,\varepsilon\,|x-y|^2} d\gamma(x)\,d\gamma(y)\le\CGS\,.
\ee
\end{theorem}
\noindent The constant $\Cst$ depends only on $\CGS$ and $\epsilon$, but not directly on the dimension $d$. This is consistent with~\cite{chen2021dimension}. In the special case of compactly supported functions, $\Cst$ can be estimated explicitly as a function of the size of the support. 

To the best of our knowledge, it is the first time such that Condition~\eqref{gautail} appears in the study of stability of~\eqref{LSIg}. Simpler conditions can also be given: for instance,~\eqref{gautail} holds true for any function $u\in\mathcal H$ such that
\be{si}
\irdg{e^{2\,\varepsilon\,|x|^2}\,|u|^2}\le\sqrt{\CGS}\,,
\ee
as a consequence of the estimate $e^{\varepsilon\,|x-y|^2}\le e^{2\,\varepsilon\,|x|^2}\,e^{2\,\varepsilon\,|y|^2}$.

\medskip Our strategy goes as follows. Under Condition~\eqref{gautail}, a density $h(t=0,\cdot)=|u|^2$ evolved by the \emph{Ornstein--Uhlenbeck} flow on $\R^d$,
\be{OU}
\frac{\partial h}{\partial t}=\L h\,,\quad(t,x)\in\R^+\times\R^d\,,
\ee
where $\L h:=\Delta h-x\cdot\nabla h$ denotes the \emph{Ornstein--Uhlenbeck} operator, is such that the measure $d\mu_T=h(T,\cdot)\,d\gamma$ satisfies a Poincar\'e inequality after some finite time $T\ge0$, as a consequence of the results of H.-B.~Chen, S.~Chewi, and J.~Niles-Weed in~\cite{chen2021dimension}.  According to M.~Fathi, E.~Indrei, and M.~Ledoux in~\cite{Fathi_2016}, this measure also satisfies an improved version of~\eqref{LSIg}. Using a \emph{backward in time} argument based on the \emph{carr\'e du champ} method, we deduce an improved~\eqref{LSIg} inequality for the initial density $|u|^2$ from the improved~\eqref{LSIg} inequality for the measure $d\mu_T$. This completes the sketch of the proof of Theorem~\ref{Thm:Main}. 

\medskip This paper is organized as follows. The remainder of the introduction is dedicated to a partial review of the literature. In Section~\ref{Sec:Prelim}, we discuss the distinction between improved inequalities and stability results, collect various observations, give stability results for which we do not claim much originality (but with new proofs and simple, explicit constants) and state the key \emph{backward in time} argument. Section~\ref{Sec:Proof} is devoted to the proof of Theorem~\ref{Thm:Main}. In Section~\ref{Sec:More}, we state an additional stability result for log-concave functions and establish a dimension-free estimate for compactly supported functions that arises as a consequence of Theorem~\ref{Thm:Main}.

\subsection{A partial review of the literature}\label{Sec:review}

In 1975, the \emph{Gaussian logarithmic Sobolev inequality}~\eqref{LSIg} was shown in~\cite{Gross75} by L.~Gross to be equivalent to the hypercontractivity of the Ornstein–Uhlenbeck semigroup. A scale invariant (but dimension-dependent) version of the Euclidean form of the inequality appeared in~\cite[Theorem~2]{MR479373}, which was already known from~\cite[Inequality~(2.3)]{MR0109101} in dimension $d=1$. By using a duality argument,~\eqref{LSIg} is also related to a Keller type estimate, see~\cite{Federbush}. The reader interested in further historical details is invited to refer to~\cite[Section~1.3.2]{Villani2008} and also to~\cite{zbMATH01503413,MR3034582,MR3339594} for further background references in information theory. More references can also be found in~\cite{Dolbeault_2015,Fathi_2016}. The optimality case in the inequality has been characterized in~\cite{MR1132315}, but can also be deduced from~\cite{bakry1985diffusions}. The logarithmic Sobolev inequality can be seen as a limit case of a family of Gagliardo-Nirenberg-Sobolev inequalities, as observed in~\cite{MR1940370} in the Euclidean setting, or as a large dimension limit of the Sobolev inequality according to~\cite{MR1164616} (see also~\cite{https://doi.org/10.48550/arxiv.2302.03926} for detailed computations and further references). For books on~\eqref{LSIg}, we refer to~\cite{ane2000inegalites,Guionnet2002,MR2352327,MR3155209}.

In a classical result on stability in functional inequalities, G.~Bianchi and H.~Egnell proved in~\cite{MR1124290} that the deficit in the Sobolev inequality measures the $\dot{\mathrm H}^1(\R^d,dx)$ distance to the manifold of the optimisers. The estimate has been made constructive in~\cite{https://doi.org/10.48550/arxiv.2209.08651}, where a new $\mathrm L^2(\R^d,d\gamma)$ stability result for the logarithmic Sobolev inequality is also established (also see~\cite{Indrei_2023,MR4305006} for further negative and positive results in other norms, for instance in strong norms like $\mathrm H^1(\R^d,d\gamma)$). To our knowledge, the first result of stability for the logarithmic Sobolev inequality is a reinforcement of the inequality due to E.~Carlen in~\cite{MR1132315} where he introduces an additional term involving the Wiener transform. An improved~\eqref{LSIg} inequality appeared in~\cite{Fathi_2016} based on a Mehler formula for the Ornstein--Uhlenbeck semigroup, which gives deficit estimates in various distances for functions inducing a Poin\-car\'e inequality. It is a result in the spirit of (b) and~\cite{Fathi_2016} is crucial for our proof of Theorem~\ref{Thm:Main}. Under the condition $\nrmG{x\,u}2=\sqrt d$, a stability result measured by a relative Fisher information is also given in~\cite{Dolbeault_2015}, on the basis of simple scaling properties of the Euclidean form of the logarithmic Sobolev inequality. For sequential stability results in strong norms, we refer to~\cite{MR4305006} when assuming a bound on~$u$ in $\mathrm L^4(\R^d,d\gamma)$ and to~\cite{Indrei_2023} when assuming a bound on $|x|^2\,u$ in $\mathrm L^2(\R^d,d\gamma)$. An earlier sequential stability result in $\mathrm L^2(\R^d,d\gamma)$ based on a Fourier technique was also given in~\cite{MR4305006} under Condition~\eqref{si}. Stability according to other notions of distance has been studied in~\cite{MR3665794,MR3320893,MR3666798}. Stability results where the distance to~$\mathscr M$ appears with a non-optimal exponent are known for instance from~\cite[Theorem~1.1]{MR3269872} where it is deduced from the HWI inequality due to F.~Otto and C.~Villani~\cite{MR1760620}. Such estimates have even been refined in~\cite{MR4116725}. There are now several other proofs. Various stability results have also been proved in Wasserstein's distance: we refer to~\cite{MR3271181,MR3269872,Fathi_2016,MR4305006,MR4455233,bolley2018dimensional,MR4116725,Indrei_2023} and in particular to~\cite[Theorem 1.3]{Indrei2024} for result with sharp exponent under an exponential moment condition. Stability in logarithmic Sobolev inequality can be related to deficit in Gaussian isoperimetry and we refer to~\cite{MR3269872} for an introduction to early results in this direction,~\cite{MR3630285} for a sharp, dimension-free quantitative Gaussian isoperimetric inequality, and~\cite{MR4116725} for recent results and further references. 

In this paper, we carefully distinguish (a) stability results where a distance to~$\mathscr M$ is controlled by the deficit, and (b) improved inequalities under appropriate constraints. Even if functions are normalized and centered, this is in some cases not enough for obtaining improved inequalities as shown in~\cite{Indrei_2023}. In fact, many counter-examples to stability are known, involving Wasserstein's distance for instance in~\cite{MR3271181,MR3269872,Fathi_2016,MR4305006,MR4455233}, weaker distances like $p$-Wasserstein, or stronger norms like $\mathrm L^p$ or $\mathrm H^1$: see for instance~\cite{MR4305006,Indrei_2023}. The classical counter-examples that apply to our setting are those of~\cite[Theorem~1.3]{MR4455233} and~\cite[Theorem~4]{MR4116725} but, as already noted in~\cite{MR4305006}, they are based on the fact that the second moment diverges along a sequence of test functions. In case of Theorem~\ref{Thm:Main}, this is forbidden by Assumption~\eqref{gautail}.

A large part of our intuition comes from the fact that the heat flow (or the Orn\-stein--Uhlenbeck flow) preserves log-concavity: see, \emph{e.g.},~\cite{zbMATH03522267,saumard2014log}. Log-concavity is a natural property in this framework for several reasons, see for instance~\cite{courtade2018quantitative}. By J.~Cheeger's inequality~\cite{MR0402831}, from log-concavity follows an explicit Poincar\'e inequality that we can use to establish an improved form of~\eqref{LSIg}. Also see~\cite{cattiaux2021functional} and references therein. However, what really matters is the Poincar\'e inequality, as noted by M.~Fathi, E.~Indrei and M.~Ledoux in~\cite{Fathi_2016}. Condition~\eqref{gautail} comes from the result of~\cite{chen2021dimension} obtained by H.-B.~Chen, S.~Chewi, and J.~Niles-Weed: such a Poincar\'e inequality holds under the evolution of the Ornstein--Uhlenbeck flow, after some explicit delay. The study of Poincar\'e inequalities evolved under~\eqref{OU} (or equivalently under convolutions with Gaussian measures) has many applications, \emph{cf.}~\cite{chen2021dimension}. The study of this question has been initiated in~\cite{zbMATH06261774,zbMATH06591637} and further investigated in~\cite{zbMATH06598355,Bardet_2018,chen2021dimension}. 

In~\cite[Theorem~1]{Fathi_2016}, the assumption is that $|u|^2d\gamma$ satisfies a Poincar\'e inequality. According to~\cite[Equation~(4.3)]{bobkov1997poincare}, this also implies the \emph{exponential moment} condition $\irdg{e^{\theta\,|x|}\,u^2}<\infty$, for some $\theta>0$. Condition~\eqref{gautail} is stronger than the \emph{exponential moment} condition, the classical example being the mesure $|u(x)|^2d\gamma(x)=e^{-|x|}\, dx$ which satisfies a Poincar\'e inequality but not~\eqref{gautail}. On the other hand,~\eqref{gautail} is, in some cases, less restrictive than the assumption of~\cite[Theorem~1]{Fathi_2016}, for instance, in the case of a compactly supported function $u$ with several disconnected components. Whether Theorem~\ref{Thm:Main} can be extended, eventually with additional restrictions, to the case of an \emph{exponential moment} condition is, to our knowledge, an open question.

Bounds on the Poincar\'e constant of a probability measure may depend on the dimension and degenerate for large dimensions according to~\cite{cattiaux2021functional,chen2021almost}. As a consequence, the same issue arises for the improvement of~\eqref{LSIg} of~\cite[Theorem 1]{Fathi_2016}. For strongly log-concave measures $u^2\,d\gamma$, one has a Poincar\'e inequality with a dimen\-sion-independent estimate of the constant according to~\cite{bakry1985diffusions}, but this class falls into the much wider class considered in Theorem~\ref{Thm:Main}.

\newpage\section{Preliminary observations, an important tool and some simple consequences}\label{Sec:Prelim}

\subsection{A discussion on stability estimates and improved inequalities}\label{Sec:Stab}

Let us define the \emph{deficit} functional of~\eqref{LSIg} by
\[
\delta[u]:=\nrmG{\nabla u}2^2-\frac12\irdg{|u|^2\,\log\(\frac{|u|^2}{\nrmG u2^2}\)}\,.
\]
The goal of \emph{stability estimates} is to find a notion of distance $\mathsf d$, an explicit constant $\beta>0$ and an explicit exponent $\alpha>0$ such that
\be{SuperStab}
\delta[u]\ge\beta\,\inf_{w\in\mathscr M}\mathsf d(u,w)^\alpha\quad\forall\,u\in\mathcal H\,.
\ee
It is known from~\cite[Corollary~1.2]{https://doi.org/10.48550/arxiv.2209.08651} that for some explicit, dimension-independent constant $\beta>0$, one has
\be{DEFFL}
\delta[u]\ge\beta\,\inf_{w\in\mathscr M}\,\nrmG{u-w}2^2\quad\forall\,u\in\mathrm H^1(\R^d,d\gamma)\,.
\ee
In this paper we consider \emph{improved inequalities} in the form
\be{stab}
\delta[u]\ge\beta\,\mathsf d(u,1)^\alpha\quad\forall\,u\in\mathcal H\,.
\ee
Any estimate of $\alpha$ and $\beta$ for~\eqref{stab} is also an estimate for~\eqref{SuperStab}, as $\inf_{w\in\mathscr M}\mathsf d(u,w)\le\mathsf d(u,1)$, because $w\equiv1\in\mathscr M$. When $\mathsf d(u,w)=\nrmG{u-w}2$, $\alpha=2$, and $\beta$ as in~\eqref{DEFFL}, Inequalities~\eqref{SuperStab} and~\eqref{stab} are in fact equivalent if $u$ is nonnegative, normalized and centred as proven in~\cite[Lemma~1]{Indrei_2023}. The equivalence between~\eqref{SuperStab} and~\eqref{stab} does not hold for $\mathsf d(u,w)=\|\nabla u-\nabla w\|_{\mathrm L^2(\R,d\gamma)}$ because the best possible exponent $\alpha$ in~\eqref{SuperStab} and~\eqref{stab} differ, as the following example shows. Assume that $d=1$ and consider the functions
\be{ex1}
u_\varepsilon(x)=1+\varepsilon\,x\quad\forall\,x\in\R^d
\ee
in the limit as $\varepsilon\to0$. With $\mathsf d(u,w)=\|u'-w'\|_{\mathrm L^2(\R,d\gamma)}$, we have
\[
d(u_\varepsilon,1)^2=\|u_\varepsilon'\|_{\mathrm L^2(\R,d\gamma)}^2=\varepsilon^2\quad\mbox{and}\quad\delta[u_\varepsilon]=\tfrac12\,\varepsilon^4+O\big(\varepsilon^6\big)\quad\mbox{as}\quad\varepsilon\to0\,.
\]
Hence, the best we can hope for in~\eqref{stab} written with $w_u=1$ is $\beta=1/2$ and $\alpha=4$. On the other hand, using the test function $w_{a_\varepsilon,c_\varepsilon}\in\mathscr M$ where $a_\varepsilon=2\,\varepsilon$ and $\log c_\varepsilon=-\,a_\varepsilon^2/4$, we obtain
\[
\inf_{w\in\mathscr M}\mathsf d(u_\varepsilon,w)^2\le d\(u_\varepsilon,w_{a_\varepsilon,c_\varepsilon}\)^2=\tfrac12\,\varepsilon^4+O\big(\varepsilon^6\big)=\delta[u_\varepsilon]+O\big(\varepsilon^6\big)\quad\mbox{as}\quad\varepsilon\to0\,,
\]
which would allow for $\beta=1$ and $\alpha=2$ in~\eqref{SuperStab}. Up to a Gaussian Poincar\'e inequality, this is compatible with the fact that~\eqref{DEFFL} still yields a stability estimate with $\alpha=2$, which is the smallest possible exponent: see~\cite[Theorem~2]{Indrei_2023}. On the other hand, the example of $u_\varepsilon$ given by~\eqref{ex1} suggests that, for non-centred functions, $\alpha=4$ is the best possible exponent in~\eqref{stab} for the distance $\mathsf d(u,w)=\|\nabla u-\nabla w\|_{\mathrm L^2(\R,d\gamma)}$.  

To establish improved inequalities with strong notions of distance $\mathsf d$, one needs, unlike in~\eqref{SuperStab}, an \emph{additional condition}. 
In the case of $\mathsf d(u,w)=\nrmG{\nabla u-\nabla w}2$, at least a control on the second-order moment $\mathsf K=\irdg{|x|^2\,u^2}$ is necessary: in~\cite{MR4305006,MR4455233}, for any $\mathsf K>d$, the authors build sequences $(u_\varepsilon)_{\varepsilon>0}$ of functions in $\mathcal H$ which satisfy~\eqref{Moment} such that
\[
\lim_{\varepsilon\to0_+}\delta[u_\varepsilon]=0\quad\mbox{and}\quad\lim_{\varepsilon\to0_+}\mathrm W_2^2\(u_\varepsilon^2\,\gamma,\gamma\)=\frac12\,(\mathsf K-d)
\]
where $\mathrm W_2$ denotes the $2$-Wasserstein distance. By~\eqref{LSIg} and the Talagrand inequality,
\[
2\,\nrmG{\nabla u_\varepsilon}2^2\ge\irdg{|u_\varepsilon|^2\,\log|u_\varepsilon|^2}\ge\frac12\,\mathrm W_2^2\(u_\varepsilon^2\,\gamma,\gamma\)\,,
\]
so that~\eqref{stab} cannot hold along such a sequence with a distance $\mathsf d$ that controls $\mathrm W_2$. Moreover, in~\cite[Theorem~1]{Indrei_2023}, E.~Indrei proves that, for any sequence $(u_n)_{n\in\N}$ of functions in $\mathcal H$ such that~\eqref{Moment} holds, one can deduce that $\lim_{n\to+\infty}\nrmG{\nabla u_n}2=0$ from $\lim_{n\to+\infty}\delta[u_n]=0$ if and only if the condition $\lim_{n\to+\infty}\nrmG{x\,u_n}2=\sqrt d$ is satisfied. Nevertheless, several results are available under a second moment condition~\cite{MR3269872,Dolbeault_2015}, fourth moment condition~\cite{Indrei_2023} and in the class of probability densities which satisfy a Poincar\'e inequality~\cite{Fathi_2016}. See also Corollary~\ref{Cor1} for a result under a second moment condition.

Let us comment on our main result Theorem~\ref{Thm:Main}. We aim at results in the strongest possible notion of distance, \emph{i.e.}, $\mathsf d(u,w)=\|u-w\|_{\mathrm H^1(\R^d,d\gamma)}$ where $w\in\mathscr M$. From Theorem~\ref{Thm:Main} and by the Gaussian Poincar\'e inequality, we find
\[
\delta[u]\ge\eta\,\nrmG{\nabla u}2^2\ge\frac\eta2\,\|u-1\|^2_{\mathrm H^1(\R^d,d\gamma)}
\]
for all $u\in\mathrm H^1(\R^d,d\gamma)$ satisfying~\eqref{Moment} and~\eqref{gautail}: this proves~\eqref{stab} for $\alpha=2$ and $\beta=\eta/2$. Notice that by two-homogeneity of $\delta$, $\alpha=2$ is the best possible exponent. Assumption~\eqref{Moment} is crucial, as illustrated by the functions $u_\varepsilon$ defined by~\eqref{ex1}. Condition~\eqref{gautail} is sharp in the following sense: for any $0<\varepsilon<1/2$, there exists a sequence $(u_{\varepsilon, n})_{n\in\N}$ of functions in~$\mathcal H$ built as in~\cite{MR4305006,MR4455233} and satisfying~\eqref{Moment} such that
\[
\lim_{n\to\infty}\irdg{e^{2\,\varepsilon\,|x|^2}\,|u_{\varepsilon,n}|^2}=+\infty\,,\quad\liminf_{n\to\infty}\nrmG{\nabla u_{\varepsilon,n}}2^2>0\,
\;\mbox{and}\;\lim_{n\to\infty}\delta[u_{\varepsilon,n}]=0\,,
\]
so that $\lim_{n\to\infty}\delta[u_{\varepsilon,n}]/\nrmG{\nabla u_{\varepsilon,n}}2^2=0$. How to fill the gap between a control on the second-order moment, which is a necessary for~\eqref{stab}, and the much more restrictive condition of Theorem~\ref{Thm:Main}, is an open question.

\subsection{The Ornstein--Uhlenbeck equation and the \emph{carr\'e du champ} method}\label{Sec:OU}

Let us recall some classical results on~\eqref{OU}. If $h_0\in\mathrm{L}^1(\R^d,d\gamma)$ is nonnegative, then there exists a unique nonnegative weak solution to~\eqref{OU} (see for instance~\cite{MR2597943}). The two key properties of the Ornstein--Uhlenbeck operator $\L h=\Delta h-x\cdot\nabla h $ are
\[
\irdg{h_1\,(\L h_2)}=-\irdg{\nabla h_1\cdot\nabla h_2}\quad\mbox{and}\quad[\nabla,\L]\,h=-\,\nabla h\,.
\]
As a consequence, we obtain the two identities
\be{Id1}
\irdg{(\L h)^2}=\irdg{\|\mathrm{Hess}\,h\|^2}+\irdg{|\nabla h|^2}
\ee
and
\be{Id2}
\irdg{\L h\,\frac{|\nabla h|^2}h}=-\,2\irdg{\mathrm{Hess}\,h:\frac{\nabla h\otimes\nabla h}h}+\irdg{\frac{|\nabla h|^4}{h^2}}\,,
\ee
where $\mathrm{Hess}\,h=(\nabla\otimes\nabla) h$ is the \emph{Hessian matrix} of $h$. Here we use the following notations. If $a$ and $b$ take values in $\R^d$, $a\otimes b$ denotes the matrix $(a_i\,b_j)_{1\le i,j\le d}$. With matrix valued $m=(m_{i,j})_{1\le i,j\le d}$ and $n=(n_{i,j})_{1\le i,j\le d}$, we define $m:n=\sum_{i,j=1}^dm_{i,j}\,n_{i,j}$ and $\|m\|^2=m:m$. If $h$ is a nonnegative solution of~\eqref{OU}, we also notice that $v=\sqrt h$ solves
\be{OUv}
\frac{\partial v}{\partial t}=\L v+\frac{|\nabla v|^2}v\,.
\ee
Let us fix $\nrmG v2=1$. The \emph{entropy} and the \emph{Fisher information}, respectively defined by
\[
\E[v]:=\irdg{|v|^2\,\log|v|^2}\quad\mbox{and}\quad\I[v]:=\irdg{|\nabla v|^2}\,,
\]
evolve along the flow~\eqref{OUv} according to
\[
\frac d{dt}\E[v(t,\cdot)]=-\,4\,\I[v(t,\cdot)]\quad\mbox{and}\quad\frac d{dt}\I[v(t,\cdot)]=-\,2\irdg{\((\L v)^2+\L v\,\frac{|\nabla v|^2}v\)}
\]
if $v$ solves~\eqref{OUv}. Using~\eqref{Id1} and~\eqref{Id2}, we obtain the classical expression of the \emph{carr\'e du champ} method
\be{cdc}
\frac d{dt}\I[v(t,\cdot)]+2\,\I[v(t,\cdot)]=-\,2\irdg{\left\|\mathrm{Hess}\,v-\frac{\nabla v\otimes\nabla v}v\right\|^2}
\ee
as for instance in~\cite{MR1842428,MR2435196,MR3155209}. By writing that
\[
\frac d{dt}\delta[v(t,\cdot)]\le0\quad\mbox{and}\quad\lim_{t\to+\infty}\delta[v(t,\cdot)]=0\,,
\]
we recover the standard proof of the \emph{entropy -- entropy production} inequality~\eqref{LSIg}, by the \emph{carr\'e du champ} method of~\cite{bakry1985diffusions}.

Several of the above expression can be rephrased in terms of the \emph{pressure variable}
\[
P:=-\,\log h=-\,2\,\log v
\]
using the following elementary identities
\begin{align*}
&\nabla v=-\,\frac12\,\sqrt h\,\nabla P\,,\quad\frac{\nabla v\otimes\nabla v}v=\frac14\,\sqrt h\,\nabla P\otimes\nabla P\,,\\
&\mathrm{Hess}\,v=-\,\frac12\,\sqrt h\,\mathrm{Hess}\,P+\frac14\,\sqrt h\,\nabla P\otimes\nabla P\,,
\end{align*}
so that, by taking into account $v\,\nabla P=-\,2\,\nabla v$ and $h=v^2$, we have
\[
\I[v]=\frac14\irdg{|\nabla P|^2\,h}\quad\mbox{and}\;\irdg{\left\|\mathrm{Hess}\,v-\frac{\nabla v\otimes\nabla v}v\right\|^2}=\frac14\irdg{\|\mathrm{Hess}\,P\|^2\,h}\,.
\]

\subsection{A backward in time estimate}\label{Sec:Backward}

Although rather elementary, this key tool of our paper is based on the following observation: if a solution $h$ of~\eqref{OU} is such that~\eqref{LSIg} holds for $v=\sqrt h$ with an improved constant at some time $T>0$, then this is also the case for the initial datum. A similar approach was used in~\cite[Lemma~2.9]{bonforte2020stability} in the case of a different flow.
\begin{lemma}\label{Lem:Backward}
Let $u\in\mathcal H$ be such that~\eqref{Moment} holds and consider the solution $v$ of~\eqref{OUv} with initial datum $u$. If for some $T>0$ we have $\delta[w]\ge c\,\nrmG{\nabla w}2^2$ with $w:=v(T,\cdot)$ for some $c \in (0,1)$, then
\[
\delta[u]\ge c\,e^{-2T}\nrmG{\nabla u}2^2\,.
\]
\end{lemma}
\noindent Notice that $c\ge1$ is impossible unless $c=1$ and $w$ is a constant.
\begin{proof} The~\eqref{LSIg} ensures that $\mathscr Q(t):=\I\left[v(t,\cdot)\right]/\E\left[v(t,\cdot)\right]\ge1/2$ for all $t\ge0$. By our assumption, $(1-c)\,\mathscr Q(T)\ge1/2$ and we learn from~\eqref{cdc} that
\[
\frac{d\mathscr Q}{dt}\le2\,\mathscr Q\,(2\,\mathscr Q-1)\,.
\]
The conclusion follows from an integration on $(0,T)$, which shows that
\[
\mathscr Q(0)\ge\frac12\,\frac1{1-e^{-2T}\,c}\,.
\]
\vspace*{-18pt}\end{proof}

\subsection{Some simple stability estimates}\label{Sec:Simple}

In this section, we collect various stability estimates and provide new proofs or explicit estimates which are new. We put the emphasis on the use of the Ornstein--Uhlenbeck equation~\eqref{OU} and on the improvements based on the \emph{carr\'e du champ} method.

\subsubsection{Improvements under moment constraints}\label{Sec:Moment}

In standard computations of the \emph{carr\'e du champ} method, one usually drops the Hessian terms, as those in right-hand side of~\eqref{cdc}. Keeping track of the remainder terms provides us with improvements as shown in~\cite{MR2152502,Demange-PhD,MR3103175} in various interpolation inequalities but generically fails in the case of the logarithmic Sobolev inequality. We remedy to this issue by introducing moment constraints.
\begin{lemma}\label{lem1} If $\sqrt h=v\in\mathcal H\cap\mathrm H^2(\R^d,d\gamma)$ is such that~\eqref{Moment} holds and $x\,v\in\mathrm L^2(\R^d,d\gamma)$, then
\[
4\,\I[v]\le\sqrt{d\irdg{\|\mathrm{Hess}\,P\|^2\,h}}+\irdg{|x|^2\,h}-d\,.
\]
\end{lemma}
\begin{proof} Using $h\,\nabla P=-\,\nabla h$, we obtain
\[
4\,\I[v]=\irdg{|\nabla P|^2\,h}=-\irdg{\nabla P\cdot\nabla h}=\irdg{h\,(\L P)}\,.
\]
After recalling that $\L P=\Delta P-x\cdot\nabla P$, we deduce that
\[
-\irdg{h\,x\cdot\nabla P}=\irdg{x\cdot\nabla h}=\irdg{h\(|x|^2-d\)}=\irdg{|v|^2\(|x|^2-d\)}
\]
using an integration by parts, which proves
\[
4\,\I[v]\le\irdg{(\Delta P)\,h}\,.
\]
The use of the Cauchy-Schwarz inequality and the arithmetic-geometric inequality
\[
(\Delta P)^2\le d\,\|\mathrm{Hess}\,P\|^2
\]
completes the proof.
\end{proof}
With the estimate of Lemma~\ref{lem1} on $\I[v]$, we have the following~result.
\begin{corollary}\label{Cor1} Let $\Psi(s):=s-\frac d4\,\log\(1+\frac4d\,s\)$. For all $u\in\mathcal H$ such that~\eqref{Moment} holds and $\nrmG{x\,u}2^2\le d$, we have the stability estimate
\be{stabsq2}
\delta[u]\ge\Psi\(\nrmG{\nabla u}2^2\)\,.
\ee
\end{corollary}
\noindent Notice that $\Psi$ is such that $\Psi(s)=\frac2d\,s^2+o(s^2)$ as $s\to0_+$, so that $\alpha=4$ is the minimal admissible exponent in Inequality~\eqref{stab}, at least for results obtained by this method.
\begin{proof}[Proof of Corollary~\ref{Cor1}] Let $h=v^2$ be the solution of~\eqref{OU} with initial datum $h_0=u^2$. Since $x\mapsto\big(|x|^2-d\big)$ is an eigenfunction of $\L$ with corresponding eigenvalue $-\,2$ and~$\L$ is self-adjoint on $\mathrm L^2(\R^d,d\gamma)$, we have that $\K(t):=\irdg{\big(|x|^2-d\big)\,h}$ satisfies
\be{MomentEvol}
\frac{d\K}{dt}=\irdg{\big(|x|^2-d\big)\,(\L h)}=\irdg{h\,\L\big(|x|^2-d\big)}=-\,2\,\K\,.
\ee
The sign of $t\mapsto\K(t)$ is conserved and in particular we have that $\irdg{|x|^2\,|v|^2}\le d$ for any $t\ge0$. For any $i=1$, $2$\ldots $d$, we also notice that $x\mapsto x_i$ is also an eigenfunction of $\L$ with corresponding eigenvalue $-\,1$ so that
\[
\frac d{dt}\irdg{x\,h}=-\irdg{x\,h}
\]
and, as a consequence $\irdg{x\,h(t,\cdot)}=0$ for all $t\ge0$ because $\irdg{x\,h_0}=0$.

For smooth enough solutions, we deduce from Lemma~\ref{lem1},~\eqref{cdc} and~\eqref{LSIg} that
\[
\frac d{dt}\I[v(t,\cdot)]+2\,\I[v(t,\cdot)]\le-\,\frac8d\,\I^2[v(t,\cdot)]\le\frac1{2\,d}\,\frac d{dt}\big(\E[v(t,\cdot)]\big)^2
\]
if $v$ solves~\eqref{OUv}. The fact that $\lim_{t\to+\infty}\I[v(t,\cdot)]=0$ follows from a Gronwall estimate relying on $\frac d{dt}\I[v(t,\cdot)]\le-\,2\,\I[v(t,\cdot)]$ and $\lim_{t\to+\infty}\E[v(t,\cdot)]=0$ is obtained as a consequence of~\eqref{LSIg}. Since $t\mapsto\I[v(t,\cdot)]-\frac12\,\E[v(t,\cdot)]-\frac1{2\,d}\,\big(\E[v(t,\cdot)]\big)^2$ is monotone nonincreasing with limit $0$ as $t\to+\infty$, we conclude that it is nonnegative for any $t\ge0$ and, as a special case, at $t=0$, thus proving that
\be{Improved1}
\delta[u]=\I[u]-\frac12\,\E[u]\ge\frac1{2\,d}\big(\E[u]\big)^2\,.
\ee

A better estimate goes as follows. Let
\[
\Phi(s):=\frac d4\(e^{\frac2d\,s}-1\)\quad\forall\,s\ge0\,.
\]
Using $\frac d{dt}\E[v(t,\cdot)]=-\,4\,\I[v(t,\cdot)]$, we notice that
\[
\frac d{dt}\Big(\I[v(t,\cdot)]-\Phi\big(\E[v(t,\cdot)]\big)\Big)\le-\,\frac8d\Big(\I[v(t,\cdot)]-\Phi\big(\E[v(t,\cdot)]\big)\Big)\, \I[v(t,\cdot)].
\]
As before, we know that
\[
\lim_{t\to+\infty}\Big(\I[v(t,\cdot)]-\Phi\big(\E[v(t,\cdot)]\big)\Big)=0\,.
\]
Moreover, Gronwall estimates show that $\I[v(t,\cdot)]-\Phi\big(\E[v(t,\cdot)]\big)$ cannot change sign and an asymptotic expansion as $t\to+\infty$ as in~\cite[Appendix~B.4]{https://doi.org/10.48550/arxiv.2211.13180} is enough to obtain that $\I[v(t,\cdot)]-\Phi\big(\E[v(t,\cdot)]\big)$ takes nonnegative values for $t>0$ large enough. Altogether, we conclude that
\[
\I[v(t,\cdot)]-\Phi\big(\E[v(t,\cdot)]\big)\ge0
\]
for any $t\ge0$ and, as a particular case, at $t=0$ for $v(0,\cdot)=u$. This provides us with the estimate
\[
\nrmG{\nabla u}2^2\ge\frac d4\(\,e^{\frac2d\irdg{|u|^2\,\log|u|^2}}-1\)\,.
\]
In the general case, one can get rid of the $\mathrm H^2(\R^d,d\gamma)$ regularity of Lemma~\ref{lem1} by a standard approximation scheme, which is classical and will not be detailed here.

As in~\cite{https://doi.org/10.48550/arxiv.2211.13180}, an estimate in a stronger norm is achieved as follows.  The function $\Phi$ is convex increasing and, as such, invertible, so that we can also write
\[
\Phi^{-1}\big(\I[u]\big)-\E[u]\ge0\,.
\]
This completes the proof of~\eqref{stabsq2} with the convex monotone increasing function
\[
\Psi(s):=s-\frac12\,\Phi^{-1}(s)\quad\forall\,s\ge0\,.
\]
\end{proof}
As far as we know, the above proof of Corollary~\ref{Cor1} is new, but the result was known by other methods, as explained in Section~\ref{Sec:Stab}. An inequality in the spirit of~\eqref{stabsq2} is proved in~\cite[Theorem~1.1 and Inequality~(1.8)]{MR3269872} using probabilistic tools. Inequality~\eqref{stabsq2} is also established in~\cite{Dolbeault_2015} using the scaling properties of the Euclidean version of~\eqref{LSIg} under the more restrictive second moment condition $\nrmG{x\,u}2^2=d$. In dimension $d=1$ with a second moment finite but arbitrarily large, a stability result in a weak norm can be found in~\cite[Theorem 1.2]{Indrei2024}. If $\K[u]<0$,~\emph{i.e.},
\[
\nrmG{x\,u}2^2<d\,,
\]
further improvements can be achieved. Here is an example. With
\[
\J[u]:=\I[u]-\frac14\,\K[u]\,,
\]
we notice that~\eqref{LSIg} can be recast as
\[
\J[u]\ge\frac12\(\E[u]-\frac12\,\K[u]\)\,.
\]
Moreover, if $v$ solves~\eqref{OUv}, then
\[
\frac d{dt}\(\E[v(t,\cdot)]-\frac12\,\K[v(t,\cdot)]\)=-\,4\,\J[v(t,\cdot)]
\]
and using Lemma~\ref{lem1}, the estimate in the proof of Corollary~\ref{Cor1} becomes
\begin{align*}
\frac d{dt}\J[v(t,\cdot)]+2\,\J[v(t,\cdot)]&=\frac d{dt}\I[v(t,\cdot)]+2\,\I[v(t,\cdot)]\\
&\le-\,\frac8d\,\J^2[v(t,\cdot)]\le\frac1{2\,d}\,\frac d{dt}\(\E[v(t,\cdot)]-\frac12\,\K[v(t,\cdot)]\)^2\,.
\end{align*}
An integration from $t=0$ to $+\infty$ shows that
\[
\delta[u]\ge\frac1{2\,d}\(\E[u]-\frac12\,\K[u]\)^2\,,
\]
which is an improvement upon~\eqref{Improved1} under the assumption that $\K[u]<0$, as we know that $\E[u]\ge0$ by Jensen's inequality.

\subsubsection{A result in the framework of log-concave densities}\label{Sec:log-concavity}

We say that a measure $d\mu$ with density $e^{-\psi}$ with respect to Lebesgue's measure is a \emph{log-concave probability measure} if $\psi$ is a convex function.
\begin{lemma}\label{Lem:Cheeger}
If $d\mu$ is a log-concave probability measure such that $\int_{\R^d}|x-x_\mu|^2\,d\mu\le\mathsf K$ where $x_\mu=\int_{\R^d}x\,d\mu$, then we have the \emph{Poincar\'e inequality}
\be{Poincare1}
\int_{\R^d}|\nabla \varphi|^2\,d\mu\ge\frac1{432\,\mathsf K}\int_{\R^d}|\varphi|^2\,d\mu\quad\forall\,\varphi\in\mathrm H^1(\R^d,d\mu)\;\mbox{such that}\;\int_{\R^d}\varphi\,d\mu=0\,.
\ee
\end{lemma}
\begin{proof} Let us denote by $\lambda_1(\mu)$ the first positive eigenvalue of $-\,\L_\psi$ where $\L_\psi$ is the \emph{Ornstein--Uhlenbeck operator} $\L_\psi:=\Delta-\nabla\psi\cdot\nabla$. We learn from~\cite[Theorem~1.2]{MR1742893} and~\cite[Ineq.~(3.4)]{MR1742893} that $432\,\mathsf K\,\lambda_1(\mu)\ge1$.
\end{proof}

\begin{lemma}\label{lemma5} Let us consider a solution $h$ of~\eqref{OU} with initial datum $h_0$ and assume that $d\mu_0:=h_0\,d\gamma$ is a log-concave probability measure. Then $d\mu_t:=h(t,\cdot)\,d\gamma$ is a log-concave probability measure for all $t\ge0$. \end{lemma}
\begin{proof} The function $g:=h\,\gamma$ solves the Fokker-Planck equation
\[
\frac{\partial g}{\partial t}=\Delta g+\nabla\cdot(x\,g)
\]
and the function $f$ such that
\be{ChangeOfVariables}
f(\taus,x):=\(1+2\,\taus\)^{-\frac d2}\,g\(\frac12\,\log(1+2\,\taus),\frac x{\sqrt{1+2\,\taus}}\)\quad\forall\,(\taus,x)\in\R^+\times\R^d
\ee
solves the heat equation
\be{heat}
\frac{\partial f}{\partial \taus}=\Delta f\quad\forall\,(\taus,x)\in\R^+\times\R^d\,.
\ee
Hence $f$ can be represented using the heat kernel. According for instance to~\cite{saumard2014log,Bardet_2018}, log-concavity is preserved under convolution, which completes the proof.
\end{proof}

\begin{lemma}\label{Lem:P} If $h\in\mathrm H^1(\R^d,d\gamma)$ is such that $\irdg{x\,h}=0$ and $P=-\,\log h$ is the \emph{pressure variable}, then
\[
\irdg{\nabla P\,h}=0\,.
\]
\end{lemma}
\begin{proof} The result follows from $\irdg{\nabla P\,h}=-\irdg{\nabla h}=\irdg{x\,h}=0$.\end{proof}

Let
\be{Cstar}
\mathscr C_\star(\mathsf K)=1+ \, \frac1{432\,\mathsf K}
\ee
where $1/432\approx0.00231481$.
\begin{proposition}\label{Prop3} For all $u\in\mathcal H$ such that~\eqref{Moment} holds, $\irdg{|x|^2\,|u|^2}=\mathsf K$ and $u^2\,\gamma$ is log-concave, with $\mathscr C_\star$ defined by~\eqref{Cstar}, we have
\[
\nrmG{\nabla u}2^2-\frac12\,\mathscr C_\star\big(\max(\mathsf K,d)\big)\irdg{|u|^2\,\log|u|^2}\ge0\,.
\]
\end{proposition}
\begin{proof} The solution $h=v^2$ of~\eqref{OU} is such that $\irdg{x\,h}=0$ and Lemma~\ref{Lem:P} applies. Since $h(t,\cdot)\,\gamma$ is log-concave for any $t\ge0$ by Lemma~\ref{lemma5}, we can apply~\eqref{Poincare1} with $f=\partial P/\partial x_i$ for any $i=1$, $2$,\ldots $d$ and obtain
\[
\irdg{\|\mathrm{Hess}\,P\|^2\,h}\ge\frac1{432\,\max(\mathsf K,d)}\irdg{|\nabla P|^2\,h}
\]
because, using~\eqref{MomentEvol}, $d+\K=d+(\mathsf K-d)\,e^{-2\,t}\le\max(\mathsf K,d)$. It follows from~\eqref{cdc} that
\[
\frac d{dt}\I[v(t,\cdot)]+2\,\mathscr C_\star(\mathsf K)\I[v(t,\cdot)]\le0\,,
\]
and the stability result is obtained as in the standard proof of the \emph{entropy -- entropy production} inequality~\eqref{LSIg} by the \emph{carr\'e du champ} method, \emph{cf.}~Section~\ref{Sec:OU}.\end{proof}

\subsubsection{From compact support to log-concavity}\label{Sec:CompactSupport}

The log-concavity property becomes true under the action of the flow of~\eqref{OU} after some delay~$t_\star$ for large classes of initial data. With the notation of Lemma~\ref{lemma5}, for any $R>0$, we read from~\cite[Theorem~5.1]{lee2003geometrical} by K.~Lee and J-L.~V\'azquez that $d\mu_t=|v(t,\cdot)|^2\,d\gamma$ is log-concave for any
\be{tstarR}
t\ge t_\star(R):=\log\(\sqrt{R^2+1}\,\)
\ee
if $v$ is solves~\eqref{OUv} with a compactly supported initial datum $u$ that is supported in a ball of radius $R>0$. The reduction from~\eqref{OUv} to the heat flow~\eqref{heat} goes as in the proof of Lemma~\ref{lemma5}. As a consequence, we know that~\eqref{Poincare1} holds for any $t\ge t_\star(R)$.
\begin{corollary}\label{Cor:Compact-LeeVazquez} Let $d\ge1$ and assume that $u\in\mathcal H$ is compactly supported in a ball of radius $R>0$. Then for all $u\in\mathcal H$ such that~\eqref{Moment} holds, with $\mathscr C_\star$ defined by~\eqref{Cstar}, we have
\[
\irdg{|\nabla u|^2}\ge\frac{\mathscr C}2\irdg{|u|^2\,\log|u|^2}
\]
with
\[
\mathscr C=1+\frac{\mathscr C_\star(\mathsf K_\star)-1}{1+R^2\,\mathscr C_\star(\mathsf K_\star)}\quad\mbox{and}\quad\mathsf K_\star:=\,\max\(d,\frac{(d+1)\,R^2}{1+R^2}\)\,.
\]
\end{corollary}
\begin{proof} Corollary~\ref{Cor:Compact-LeeVazquez} follows from $t_\star(R)$ given by~\eqref{tstarR} so that $e^{-2\,t_\star(R)}=1/(1+R^2)$, Proposition~\ref{Prop3} applied at $t=t_\star(R)$ with $\mathsf K=(d+1)\,R^2/(1+R^2) \ge \K\big(t_\star(R)\big)+d$ by~\eqref{MomentEvol}, and Lemma~\ref{Lem:Backward} applied with $T=t_\star(R)$ and $c=1-1/\mathscr C_\star(\mathsf K_\star)$, so that
\[
\delta[u]=\nrmG{\nabla u}2^2-\frac12\irdg{|u|^2\,\log|u|^2}\ge\frac{\mathscr C_\star(\mathsf K_\star)-1}{\big(1+R^2\big)\,\mathscr C_\star(\mathsf K_\star)}\,\nrmG{\nabla u}2^2\,.
\]
\end{proof}

\section{Proof of Theorem \texorpdfstring{\ref{Thm:Main}}1}\label{Sec:Proof}

\noindent The proof of Theorem~\ref{Thm:Main} is based on three ingredients:
\begin{enumerate}
\item A stability result for~\eqref{LSIg} obtained by M.~Fathi, E.~Indrei, and M.~Ledoux in~\cite{Fathi_2016} for a special class of initial data satisfying a Poincar\'e inequality,
\item The result of~\cite{chen2021dimension} by H.-B.~Chen, S.~Chewi, and J.~Niles-Weed which states that after a finite time $ T\ge0$, a solution to the Ornstein--Uhlenbeck flow is in the above class under Condition~\eqref{gautail},
\item The backward in time argument based on the \emph{carr\'e du champ} method of Section~\ref{Sec:Backward}, as in~\cite{bonforte2020stability}, which is used on the interval $[0,T)$.
\end{enumerate}

\subsection{Evolving a Poincar\'e inequality by the Ornstein--Uhlenbeck flow and application}\label{Sec:Fathi-Indrei-Ledoux}
In this section we collect some information about the Poincar\'e inequality for a measure $v^2\,d\gamma$ where the function $v$ is evolving under the Ornstein-Uhlenbeck equation~\eqref{OUv}. As in~\cite{Fathi_2016}, let us define $\mathcal P(\lambda)$, for any $\lambda>0$, as the set of all functions~$u$ such that the measure $u^2\,d\gamma$ is a probability measure which satisfies a Poincar\'e inequality
\[
\irdg{|\nabla \varphi|^2\,u^2}\ge\lambda\irdg{|\varphi|^2\,u^2}\quad\forall\,\varphi\in\mathcal H_u
\]
where $\mathcal H_u$ is the space of the functions $\varphi\in\mathrm H^1(\R^d,u^2\,d\gamma)$ such that $\irdg{\varphi\,u^2}=0$. The following Lemma is a key step in obtaining~\cite[Theorem 1]{Fathi_2016} and its proof can be found in~\cite[Section 2]{Fathi_2016}.
\begin{lemma}\label{Lem:Fathi-Indrei-Ledoux}
If $v(t,\cdot)$ solves~\eqref{OUv} with initial datum $v(t=0,\cdot)=u$ such that $u\in\mathcal P(\lambda)$ for some $\lambda>0$, then for any $t\ge0$ we have the Poincar\'e inequality
\[
\irdg{|\nabla \varphi|^2\,|v(t,\cdot)|^2}\ge\frac\lambda{\lambda+(1-\lambda)\,e^{-2\,t}}\irdg{|\varphi|^2\,|v(t,\cdot)|^2}\quad\forall\,\varphi\in\mathcal H_{v(t,\cdot)}\,.
\]
\end{lemma}
\noindent Let us consider the function
\[
\sigma(\lambda):=\frac{\lambda^2-\lambda-\lambda\,\log\lambda}{(\lambda-1)^2}\quad\forall\lambda\in(0,1)\cup(1,+\infty)
\]
and extend it by $\sigma(1)=1/2$. We may notice that $\sigma$ is monotone increasing and concave, with $\lim_{\lambda\to0_+}\sigma(\lambda)=0$ and $\lim_{\lambda\to+\infty}\sigma(\lambda)=1$. An interesting consequence of Lemma~\ref{Lem:Fathi-Indrei-Ledoux} is the following stability estimate for~\eqref{LSIg}.
\begin{corollary}[\cite{Fathi_2016}]\label{Cor:Fathi-Indrei-Ledoux}
Let $\lambda>0$ and $u\in\mathcal H$ satisfying Condition~\eqref{Moment} and such that $u\in\mathcal P(\lambda)$, then we have that
\[
\delta[u]\ge\sigma(\lambda)\,\nrmG{\nabla u}2^2\,.
\]
\end{corollary}
\noindent The proof of Corollary~\ref{Cor:Fathi-Indrei-Ledoux} relies on the same strategy as in the proof of Proposition~\ref{Prop3}, except that one has to use the Poincar\'e  inequality of Lemma~\ref{Lem:Fathi-Indrei-Ledoux} to write that
\[
\irdg{\|\mathrm{Hess}\,P\|^2\,h}\ge\frac\lambda{\lambda+(1-\lambda)\,e^{-2\,t}}\irdg{|\nabla P|^2\,h}\quad\forall\,t\ge0\,.
\]
The result follows from an integration on $t\in\R^+$ of~\eqref{cdc} using the above inequality.

\subsection{Gaussian convolutions of measures through the Ornstein-Uhlenbeck flow}\label{Sec:Chen-Chewi-NilesWeed}

For any $a>1$ and $b>0$, let us define the function
\[
F(a,b):=\frac1a\(\frac a{a-1}+b^\frac1{a-1}\)^{-1}\,.
\]
The following result is a rephrasing of the main result in~\cite{chen2021dimension} using the Ornstein-Uhlenbeck flow.
\begin{lemma}\label{Lem:Chen-Chewi-NilesWeed}
Let $v(t,\cdot)$ be a solution to~\eqref{OUv} with $v(t=0,\cdot)=u\in\mathcal H$. If $u$ satisfies~\eqref{Moment} and~\eqref{gautail} for some $\varepsilon>0$ and $\CGS>0$, then
\be{admissible.lambda}
v(t,\cdot)\in\mathcal P\big(\lambda(t)\big)\quad\mbox{with}\quad\lambda(t):=\varepsilon\,e^{2\,t}\,F\(\varepsilon\(e^{2\,t}-1\),\CGS\)\quad\forall\,t\ge t_\varepsilon:=\frac12\,\log\(1+\frac1{\varepsilon}\)\,.
\ee
\end{lemma}
\begin{proof} As in the proof of Lemma~\ref{lemma5}, if the function $f$ solves~\eqref{heat} with initial datum $|u|^2\,\gamma$, then
\[
\gamma(x)\,|v(t,x)|^2=e^{dt}\,f\(\taus,y\)\quad\mbox{with}\quad \taus=\frac12\(e^{2\,t}-1\)\quad\mbox{and}\quad y=e^t\,x
\]
by~\eqref{ChangeOfVariables}. By a direct application of~\cite[Theorem 2]{chen2021dimension}, we learn that $f(\taus,\cdot)/\sqrt\gamma$ is in $\mathcal P(\Lambda)$ with $\Lambda=\varepsilon\,F(2\,\varepsilon\,\taus,\CGS)$ if $2\taus>1/\varepsilon$, \emph{i.e.}, if $t=t(\taus)\ge t_\varepsilon$. Changing variables, if $\psi(y)=\varphi\(e^{-t}\,y\)$, we have
\begin{align*}
&\irdg{|\nabla_x \varphi(x)|^2\,|v(t,x)|^2}=e^{2\,t}\irdy{ \left|\nabla\psi(y)\right|^2\,f(\taus,y)}\,,\\
&\irdg{|\varphi(x)|^2\,|v(t,x)|^2}=\irdy{\left|\psi(y)\right|^2\,f(\taus,y)}\,,
\end{align*}
for any $\varphi\in\mathrm H^1(\R^d,v^2\,d\gamma)$ with zero average with respect to the measure $v^2\,d\gamma$. Since the function $\psi$ has zero average with respect to the measure $f(\taus,y)\,dy$, the corresponding Poincar\'e inequality amounts to~\eqref{admissible.lambda} written with $\lambda=\lambda(t)=e^{2 t}\,\Lambda$. This concludes the proof. 
\end{proof}

\subsection{Conclusion}\label{ss:conclusion}

We can now complete the proof of Theorem~\ref{Thm:Main} as follows. We recall that $v(t,\cdot)$ solves~\eqref{OUv} with initial datum $v(t=0,\cdot)=u$ such that~\eqref{Moment} and~\eqref{gautail} hold. By Corollary~\ref{Lem:Fathi-Indrei-Ledoux} and Lemma~\ref{Lem:Chen-Chewi-NilesWeed}, we know that
\[
\delta[v(t,\cdot)]\ge\sigma\big(\lambda(t)\big)\,\nrmG{\nabla v(t,\cdot)}2^2\quad\forall\,t\ge t_\varepsilon
\]
with $\lambda(t)$ and $t_\varepsilon$ given by~\eqref{admissible.lambda}. As a consequence of Lemma~\ref{Lem:Backward} applied with $T=t_\varepsilon$, we have
\be{eta}
\delta[u]\ge\eta\,\nrmG{\nabla u}2^2\quad\mbox{with}\quad\eta=\sup_{t>t_\varepsilon}\sigma\big(\lambda(t)\big)\,e^{-2\,t}\,.
\ee
\qed

\noindent With the admissible choice of $t$ satisfying  $\varepsilon\,(e^{2\,t}-1) = 2$, we obtain the simple lower estimate
\[
\eta\ge\frac\varepsilon{2+\varepsilon}\,\sigma\(\frac{2+\varepsilon}{2\,(2+\CGS)}\)\,.
\]

\section{Log-concavity and compactly supported functions}\label{Sec:More}

This section collects some additional results about stability for log-concave measures and compactly supported functions. 

\subsection{Stability results for log-concave measures}

Here is a simple improvement of Proposition~\ref{Prop3} based on Corollary~\ref{Cor:Fathi-Indrei-Ledoux}.
\begin{corollary}\label{cor:cor1}
Assume $\nu:=u^2\,d\gamma$ is a logarithmically concave probability measure with $u\in\mathcal H$, such that~\eqref{Moment} is satisfied, and $\nrmG{x\,u}2^2=\mathsf K$. Then, the stability estimate
\[
\delta[u]\ge\sigma(\lambda)\irdg{|\nabla u|^2}
\]
holds with $1/\lambda=432\,\mathsf K$.
\end{corollary}
\begin{proof} This result is a simple consequence of Lemma~\ref{Lem:Cheeger} and Corollary~\ref{Cor:Fathi-Indrei-Ledoux}.\end{proof}

\subsection{Compactly supported functions}

From Corollary~\ref{Cor:Compact-LeeVazquez}, we obtain an improved~\eqref{LSIg} with a bound which depends on the dimension $d$. Dimensional dependence is a huge topic in functional inequalities and we refer to~\cite{Eskenazis_2024} for further considerations in this direction. As a final remark, let us notice that Theorem~\ref{Thm:Main} provides us with a dimension-free result.
\begin{proposition} If $u\in\mathcal H$ is supported in $B(0,R)$ and satisfies~\eqref{Moment}, then
\[
\delta[u]\ge\frac\alpha{1+R^2}\,\nrmG{\nabla u}2^2
\]
with $\alpha=\big((1+e)\,\log(1+e)-e\big)/e^2\,\approx\,0.292973$.
\end{proposition}
\begin{proof} We may notice that~\eqref{gautail} is satisfied for every $\varepsilon>0$, and $\CGS=\exp(\varepsilon\,R^2)$. The result follows by taking the limit as $\varepsilon\to\infty$ in~\eqref{eta} with $e^{2\,t}=1+R^2$ and \hbox{$\lambda(t)\ge(1+e)^{-1}$}.
\end{proof}

\section*{Acknowledgements} The authors thank Max Fathi and Pierre Cardaliaguet for fruitful discussions and Emanuel Indrei for stimulating interactions. They also thank an anonymous referee for useful comments and suggestions which have led to an improvement of the manuscript. They also want to express their gratitude to the managing editor, L.~Gross, for his encouragements and questions. G.B.~has been funded by the European Union’s Horizon 2020 research and innovation program under the Marie Sklodowska-Curie grant agreement No 754362. This work has been (partially) supported by the Project Conviviality ANR-23-CE40-0003 of the French National Research Agency\\
\noindent{\scriptsize\copyright\,2024 by the authors. This paper may be reproduced, in its entirety, for non-commercial purposes.}
\bibliographystyle{siam}\small
\bibliography{BDS2024-LSI}
\end{document}